\documentclass[12pt]{amsart}
\usepackage[margin=1in]{geometry}                
\usepackage{verbatim}
\usepackage{graphicx}
\usepackage{amssymb}
\usepackage{epstopdf}
\usepackage{natbib}
\usepackage{color}
\usepackage{subcaption,caption}
\usepackage{tikz}
\usepackage{url}
\tikzstyle{vertex}=[circle, draw, inner sep=0pt, minimum size=6pt, fill=black]
\newcommand{\vertex}{\node[vertex]}

\theoremstyle{plain}
\newtheorem{thm}{Theorem}[section]
\newtheorem{lemma}[thm]{Lemma}
\newtheorem{prop}[thm]{Proposition}
\newtheorem{cor}[thm]{Corollary}

\newtheorem*{thm*}{Theorem}
\newtheorem*{lemma*}{Lemma}
\newtheorem*{prop*}{Lemma}
\newtheorem*{cor*}{Corollary}
\newtheorem*{conj*}{Conjecture}

\theoremstyle{definition}
\newtheorem{defn}[thm]{Definition}
\newtheorem*{defn*}{Definition}
\newtheorem{ex}[thm]{Example}

\theoremstyle{remark}
\newtheorem{rmk}[thm]{Remark}

\newcommand{\zz}{\mathbb{Z}}

\newcommand{\rr}{\mathbb{R}}
\newcommand{\cc}{\mathbb{C}}
\newcommand{\kk}{\mathbb{K}}

\newcommand{\bfp}{\mathbf{p}}
\newcommand{\bfq}{\mathbf{q}}

\newcommand{\cals}{\mathcal{S}}

\newcommand{\rank}{\textnormal{rank}}

\newcommand{\puis}{\cc\{\!\{t\}\!\}}
\DeclareMathOperator\val{val}
\DeclareMathOperator\trop{trop}
\DeclareMathOperator{\cm}{CM}

\DeclareMathOperator\cip{cip}
\DeclareMathOperator\clade{clade}

\DeclareMathOperator\tp{tp}

\newcommand{\TP}{\mathbb{TP}}
\DeclareMathOperator{\gr}{Gr}
\title{The tropical Cayley-Menger variety}
\author{Daniel Irving Bernstein}
\address{Institute for Data, Systems, and Society, Massachusetts Institute of Technology, 77 Massachusetts Avenue, Cambridge, MA 02139}
\email{dibernst@mit.edu}\urladdr{https://dibernstein.github.io}
\author{Robert Krone}
\address{UC Davis Department of Mathematics, 1 Shields Avenue, Davis, CA 95616}
\email{rkrone@math.ucdavis.edu}
\urladdr{http://rckr.one/}

\begin{document}
\begin{abstract}
    The \emph{Cayley-Menger variety} is the Zariski closure of the set of vectors specifying the pairwise squared distances between $n$ points in $\rr^d$.
    This variety is fundamental to algebraic approaches in rigidity theory.
    We study the tropicalization of the Cayley-Menger variety.
    In particular, when $d = 2$, we show that it is the Minkowski sum of the set of ultrametrics on $n$ leaves with itself,
    and we describe its polyhedral structure.
    We then give a new, tropical, proof of Laman's theorem.
\end{abstract}

\maketitle


Tropicalization is a process that transforms a variety into a polyhedral complex
in a way that preserves many essential features.
One of our main results is a combinatorial description of the tropicalization of the \emph{Pollaczek-Geiringer variety},
i.e.~the Zariski closure of the set of vectors specifying the pairwise squared distances between $n$ points in $\rr^2$.
We show that this tropical variety has a simplicial complex structure that we describe in terms of pairs of rooted trees.
Another main result is a new proof of Laman's theorem from rigidity theory
via our combinatorial description of the tropicalization of the Pollaczek-Geiringer variety.

Laman's theorem can be seen as a combinatorial description of the algebraic matroid underlying the Pollaczek-Geiringer variety.
Our proof of Laman's theorem takes this viewpoint
and uses a lemma of Yu from \cite{yu2017algebraic},
saying that tropicalization preserves algebraic matroid structure.
A similar strategy was adopted by the first author in \cite{bernstein2017completion},
wherein he characterized the algebraic matroids underlying the Grassmannian $\gr(2,n)$ of planes in affine $n$-space,
and the determinantal variety of $m\times n$ matrices of rank at most two.
A key ingredient was a result of Speyer and Sturmfels \cite{speyer2004tropical}
describing the tropicalization of $\gr(2,n)$.

Loosely speaking, a graph is said to be \emph{generically rigid in $\rr^d$}
if when its vertices are embedded in $\rr^d$ at generic points
and its edges are treated as rigid struts that are free to move about the vertices,
the resulting structure cannot be continuously deformed.
\emph{Laman's theorem} is an elegant characterization of the graphs that are minimally generically rigid in $\rr^2$,
and such graphs are said to be \emph{Laman}.
In spite of the name, Laman's theorem was originally proved by Hilda Pollaczek-Geiringer in 1927 \cite{pollaczek1927gliederung},
though this was evidently forgotten when Laman rediscovered it in 1970 \cite{laman1970graphs}.
Pollaczek-Geiringer's work on this topic seems only to have resurfaced recently,
so the terms ``Laman's theorem'' and ``Laman graphs'' have become quite deeply embedded in the rigidity theory literature,
and we stick with them in this paper.

Capco, Gallet, Grasegger, Koutschan, Lubbes, and Schicho recently used tropical geometry in \cite{capco2018number}
to compute upper bounds on the number of realizations of a given Laman graph with generic prescribed edge lengths.
Perhaps the most important open problem in rigidity theory is to characterize the
graphs that are minimally generically rigid in $\rr^3$,
and no currently known technique for proving Laman's theorem seems likely to extend.
Therefore, one motivation for our tropical proof is hope that it may one day extend to the $\rr^3$ case.

In addition to being an interesting mathematical subject,
rigidity theory of graphs has diverse applications.
It can be used to discover the structure of molecules \cite{liberti2011molecular}
which is particularly useful when studying proteins \cite{jacobs2001protein,rader2002protein}
and materials at the nano scale \cite{bauchy2014nanoscale,micoulaut2017material}.
Macro-scale applications include
coordinating groups of autonomous vehicles \cite{anderson2008rigid,eren2005merging,eren2002framework,olfati2002graph}
and sensor network localization \cite{eren2004rigidity,zhu2010universal}.

We give the necessary technical background on rigidity theory and tropical geometry in Section~\ref{sec:preliminaries}.
Among other things, we define the \emph{Cayley-Menger variety},
a generalization of the Pollaczek-Geiringer variety, 
which is the Zariski closure of the set of vectors specifying the pairwise squared distances between $n$ points in $\rr^d$.
In Section~\ref{sec:tropicalCayleyMenger1},
we show that the tropicalization of the Cayley-Menger variety in the case $d = 1$
is the space of ultrametrics on $n$ leaves.
We also set some notation that will be used in later sections.
We begin Section~\ref{sec:tropicalLaman} by showing that the tropicalization of the Pollaczek-Geiringer variety
is the Minkowski sum of the set of ultrametrics with itself.
Then, we show that this tropical variety
admits a particular simplicial complex structure.
In Section \ref{sec:tropicalProofOfLaman},
we use our previous results to give a new proof of Laman's theorem.

\section{Preliminaries}\label{sec:preliminaries}

Let $\kk$ be $\rr$ or $\cc$.
Let $S$ be a finite set, and let $X \subseteq \kk^S$ be an irreducible affine variety.
Each $E \subseteq S$ defines a coordinate projection $\pi_E: \kk^S \rightarrow \kk^E$.
The \emph{algebraic matroid underlying $X$}
is the matroid on ground set $S$ whose independent sets are the $E \subseteq S$
such that $\dim \pi_E(X) = |E|$.
To see that this construction yields a matroid,
see e.g.~\cite[Proposition 1.2.9]{bernstein2018matroids}.

A \emph{bar and joint framework} consists of a graph $G = (V,E)$ along with an
injection $\bfp:~V~\rightarrow~\rr^d$.
We denote such a framework by $(V,E,\bfp)$
and say that it is \emph{rigid} if there exists an $\varepsilon > 0$
such that for any other injection $\bfq: V \rightarrow \rr^d$ satisfying
$\sum_{u \in V} \|\bfp(u) - \bfq(u)\|^2 \le \varepsilon$
and $\|\bfq(u) - \bfq(v)\| = \|\bfp(u) - \bfp(v)\|$ for all $uv \in E$,
then the images of $\bfp$ and $\bfq$ are related by a Euclidean isometry of $\rr^d$.
A graph $G = (V,E)$ is said to be \emph{generically rigid in $\rr^d$}
if every framework $(V,E,\bfp)$ is rigid when $\bfp$ is generic.
We will identify injections $\bfp: V \rightarrow \rr^d$
with point configurations in $(\rr^d)^{|V|}$.

The \emph{Cayley-Menger variety of $n$ points in $\rr^d$},
denoted $\cm_n^d$, is the affine variety
embedded in $\cc^{\binom{[n]}{2}}$ given as the Zariski closure of the
set of pairwise squared euclidean distances between
$n$ points in $\rr^d$.
When $d = 2$,
we will call the corresponding Cayley-Menger variety $\cm_n^2$ the \emph{Pollaczek-Geiringer variety}.
The following lemma gives three folklore results,
the first of which is called \emph{Laman's condition} in \cite{graver2008combinatorial}.
They are well-known, but we give proofs as they are not generally phrased in our algebraic-geometric language.

\begin{lemma}\label{lemma:indSetCM}
    Let $n \ge d$ and let $E \subseteq \binom{[n]}{d}$. Then:
    \begin{enumerate}
        \item if $E$ is independent in the algebraic matroid of $\cm_n^d$,
        then for all $V \subseteq [n]$ with $|V| \ge d$,
        the induced subgraph of $([n],E)$ on vertex set $V$ has at most $d|V| - \binom{d+1}{2}$ edges,
        \item the dimension of the Cayley-Menger variety $\cm_n^d$ is $dn - \binom{d+1}{2}$, and
        \item the graph $([n],E)$ is generically rigid if and only if $E$ is spanning in the algebraic matroid underlying $\cm_n^d$.
    \end{enumerate}
\end{lemma}
\begin{proof}
    The first statement follows from the second
    by the observation that the coordinate projection of $\cm_{n}^d$
    onto the coordinates indexed by $\binom{V}{2}$ is $\cm_{|V|}^d$.

    We now prove the second statement. Let $\phi: (\rr^{d})^n \rightarrow \rr^{\binom{[n]}{2}}$ be the map
	sending a configuration of $n$ points in $\rr^d$ to the set
	of pairwise squared distances among them.
	Then $\cm_{n}^d$ is the Zariski closure of the image of $\phi$.
    The map $\phi$ is algebraic, so
	 \[ \dim (\cm_{n}^d) = \dim((\rr^{d})^n) - \dim(\phi^{-1}(\phi(\bfp))), \]
    where $\bfp$ is a generic point configuration in $(\rr^{d})^n$.
	 
    Let $E(d)$ denote the group of euclidean isometries of $\rr^d$.  Fiber $\phi^{-1}(\phi(\bfp))$ is equal to the $E(d)$-orbit of $\bfp$.  Since $\bfp$ was chosen generically, the points affinely span $\rr^d$ if $n > d$, and so the only element of $E(d)$ that stabilizes $\bfp$ is the identity transformation.  If $n = d$ then the points in $\bfp$ affinely span a hyperplane $H \subseteq \rr^d$, and the only two elements of $E(d)$ that stabilize $\bfp$ are the identity and the reflection across $H$.  It follows from \cite[Theorem 9.24]{lee2003introduction}
	that $\dim(\phi^{-1}(\phi(\bfp)))=\dim(E(d))$.
    To see that $\dim(E(d)) = \binom{d+1}{2}$,
    note that each translation is specified by $d$ independent parameters,
    and each rotation or reflection is specified by $\binom{d}{2}$ independent parameters.

    We now prove the third statement.
    Note that $E$ is spanning in $\cm_n^d$ if and only if $\dim(\pi_E(\cm_n^d)) = \dim(\cm_n^d)$.
    Equivalently, for a generic point configuration $\bfp \in (\rr^{d})^n$,
    the set $\pi_E^{-1}(\pi_E(\phi(\bfp)))$ is zero-dimensional, i.e.~a finite set.
    Thus $\phi^{-1}(\pi_E^{-1}(\pi_E(\phi(\bfp))))$ consists of finitely many orbits of $E(d)$'s diagonal action on $(\rr^d)^n$.
    Taking $\varepsilon$ to be half the minimum distance between any two such orbits,
    we see that the framework $([n],E,\bfp)$ is rigid.
\end{proof}

For $d = 1$ and $d = 2$,
the necessary condition from Lemma \ref{lemma:indSetCM} 
for independence is known to be sufficient.
The $d = 1$ case is trivial and the $d = 2$ case is known as \emph{Laman's Theorem}.

What follows is a very brief introduction to tropical geometry.
The theory of tropical geometry can be developed using either the \emph{max convention} or \emph{min convention}.
Both give exactly the same theorems, modulo some sign changes and the substitution of ``maximum'' with ``minimum'' or vice versa.
One often chooses the convention that minimizes the number of negative signs that appear.
In this paper, that happens to be the max convention so that is what we choose.
See \cite{maclagan2015introduction} for a more detailed introduction to tropical geometry (but note that it is written in the min convention).

A \emph{valuation} on a field $K$ is a function $\val: K \to \rr \cup \{-\infty\}$ satisfying:
 \begin{enumerate}
  \item\label{val:zero} $\val(a) = -\infty$ if and only if $a = 0$,
  \item\label{val:mult} $\val(ab) = \val(a) + \val(b)$,
  \item\label{val:add} $\val(a + b) \leq \max\{\val(a),\val(b)\}$ with equality if $\val(a) \neq \val(b)$.
 \end{enumerate}
One should think of $\val(a)$ as a measure of the magnitude of $a \in K$ that behaves roughly like a logarithm, as reflected by rules (\ref{val:zero}) and (\ref{val:mult}).  If $a \in K$ has smaller valuation than $b \in K$, then $a$ should be considered insignificant compared to $b$ so $\val(a+b) = \val(b)$.  On the other hand, if $a$ and $b$ have the same valuation, adding them may cancel the largest magnitude components of each, so $\val(a+b) \leq \val(b)$, as described by rule (\ref{val:add}).
The pair $(K,\val)$ is called a \emph{valuated field}.

For our purposes, we require a valuated field $K$ that extends $\cc$, is algebraically closed, and with valuation that maps densely into $\rr\cup \{-\infty\}$.  Therefore we will take $K = \puis$, the field of complex Puiseux series.  The elements of $\puis$ are formal series in indeterminant $t$ of the form
 \[ a = \sum_{i = m}^\infty c_i t^{i/k} \]
for some integer $m$ and some positive integer $k$, with each $c_i$ in $\cc$ and $c_m \neq 0$.  The valuation is defined by $\val(a) = -m/k$, the negative of the smallest exponent of $t$.

For $Y$ an algebraic variety in $K^S$, the {\em tropicalization}, $\trop(Y)$, of $Y$ is the closure in the Euclidean topology of the image of $Y$ under the map $\val$,
\[ \trop(Y) := \overline{\{\val(x) : x \in Y\}} \cap \rr^S. \]
Note that we discard the points with coordinates $-\infty$.

To tropicalize an algebraic variety $X \subseteq \cc^S$ with defining ideal $I \subseteq \cc[x_i: i \in S]$, we extend scalars from $\cc$ to the valuated field $K$ to get a variety in $K^S$.
Let $X' \subseteq K^S$ denote the vanishing set of ideal $IK[x_i : i \in S]$.
We then define $\trop(X)$ to be equal to $\trop(X')$.  

In this case, $\trop(X)$ is a pure polyhedral fan of the same dimension as $X$ \cite{bogart2007computing}.  By studying $\trop(X)$, we can apply tools from polyhedral geometry and combinatorics to questions about $X$.  Of particular interest for this paper, the following lemma tells us that tropicalization preserves the algebraic matroid structure.

\begin{lemma}[{\cite[Lemma 2]{yu2017algebraic}}]\label{lemma:tropicalPreservesAlgebraicMatroid}
    Let $X \subseteq \cc^S$ be an irreducible variety and let $E \subseteq S$.
    Then the projection of $X$ to $\cc^E$ has the same dimension as the projection of $\trop(X)$
    to $\rr^E$.
\end{lemma}

We now review the results from the literature that we will need to obtain our combinatorial description of $\trop(\cm_n^2)$.
Recall that a monomial map $\alpha: \cc^n \rightarrow \cc^d$ is an algebraic map with the property that each coordinate of the image is given by a monomial in the coordinates of $\cc^n$.  Such a map can be represented by a matrix $A$
where the $i^{\rm th}$ column of $A$ is the exponent vector of the $i^{\rm th}$ coordinate of $\alpha$.



\begin{thm}[\cite{sturmfels2007elimination}, Theorem 1.1]\label{thm:monomialMapTropicalVariety}
    Let $A \in \zz^{n\times d}$ be an integer matrix representing a monomial map $\alpha: \cc^n\rightarrow \cc^d$
    and let $X \in \cc^n$ be a variety.
    Then $\trop(\alpha(X)) = A \trop(X)$.
\end{thm}

We denote the coordinates of points $\delta \in \rr^{\binom{[n]}{2}}$ by $\delta_{uv}$ where $u < v$.
We say that $\delta \in \rr^{\binom{[n]}{2}}$ is an \emph{ultrametric} if
$\delta_{uv}~\le~\max\{\delta_{uw},\delta_{vw}\}$ for all triples $u,v,w$ of distinct elements of $[n]$.
Note that we do not require nonnegativity of any coordinates.

We now recall the well-known way that ultrametrics can be represented on rooted trees (see e.g. \cite[Chapter 7]{semple2003phylogenetics}).
Given a rooted tree $T$ with leaves labeled by $[n]$,
the most recent common ancestor of a pair of leaves $u,v \in [n]$
is the unique internal node in the unique path in $T$
from $u$ to $v$ that is closest to the root in the graph-theoretic distance.
Given an ultrametric $\delta$ on $[n]$,
there exists a unique tree $T$,
whose internal nodes are assigned real-valued weights that increase along any path towards the root,
such that $\delta(u,v)$ is the weight assigned to the most recent common ancestor of $u$ and $v$.
Given an ultrametric $\delta$, the associated tree $T$ (disregarding the weights on the internal vertices) is called the {\em topology} of $\delta$.
See Figure \ref{fig:ultrametricExample} for an example
displaying the ultrametric $(\delta_{12},\delta_{13},\delta_{14},\delta_{23},\delta_{24},\delta_{34}) = (-2,1,4,1,4,4)$ on its topology.
We denote the set of all ultrametrics in $\rr^{\binom{[n]}{2}}$ by $U_n$.

Now, let $\sqrt{\cm^1_n} \subseteq \rr^{\binom{[n]}{2}}$ denote the linear space
parameterized by $\delta_{uv} = x_u-x_v$.
Our results all rest on the following theorem of Ardila and Klivans.

\begin{thm}[\cite{ardila2006bergman}, Theorem 3]\label{thm:ardilaKlivansBergmanFan}
    The tropicalization of the linear space $\sqrt{\cm^1_n}$ is the set of ultrametrics on $[n]$.
    That is, $\trop(\sqrt{\cm^1_n}) = U_n$.
\end{thm}

\begin{figure}[h!]
    \begin{tikzpicture}
        \node [below] (a) at (0,0){$1$};
        \node [below] (b) at (1,0){$2$};
        \node [below] (c) at (2,0){$3$};
        \node [below] (d) at (3,0){$4$};
        \node (ab) at (1/2-0.3,1/2+0.2){$-2$};
        \node (abc) at (1-0.2,1+0.2){$1$};
        \node (abcd) at (3/2-0.2,3/2+0.2){$4$};
        \draw (0,0) -- (1/2,1/2);
        \draw (1,0) -- (1/2,1/2);
        \draw (1/2,1/2) -- (1,1);
        \draw (2,0) -- (1,1);
        \draw (3,0) -- (3/2,3/2);
        \draw (1/2,1/2) -- (3/2,3/2);
        \node at (7,1){$\bordermatrix{
            & 12 & 123 & 1234\cr
            12 & 1 & 0 & 0\cr
            13 & 0 & 1 & 0\cr
            14 & 0 & 0 & 1\cr
            23 & 0 & 1 & 0\cr
            24 & 0 & 0 & 1\cr
            34 & 0 & 0 & 1
        }$};
    \end{tikzpicture}
    \caption{On the left is a rooted tree displaying the ultrametric $(-2,1,4,1,4,4)$.
    Letting $T$ denote the topology of this tree, the corresponding matrix $M_T$ is displayed on the right.
    Its columns span the linear hull of the set of ultrametrics whose topology is $T$.}\label{fig:ultrametricExample}
\end{figure}
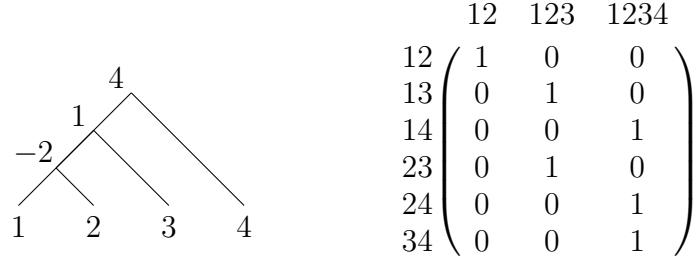

\section{The tropical Cayley-Menger variety in dimension 1}\label{sec:tropicalCayleyMenger1}

\begin{prop}\label{prop:CM1}
    The tropicalization of the Cayley-Menger variety of $n$ points in $\rr^1$ is the set of ultrametrics on $[n]$.
    That is, $\trop(\cm^1_n) = U_n$.
\end{prop}
\begin{proof}
    Let $\alpha: \cc^{\binom{[n]}{2}} \rightarrow \cc^{\binom{[n]}{2}}$ denote the monomial map that squares each coordinate.
    The matrix $A$ representing $\alpha$ is twice the identity matrix.
    Since $\cm^1_n = \alpha(\sqrt{\cm_n^1})$, Theorem \ref{thm:monomialMapTropicalVariety} implies that 
    $\trop(\cm^1_n) = A \trop(\sqrt{\cm_n^1})$.
    Theorem \ref{thm:ardilaKlivansBergmanFan} says that $\trop(\sqrt{\cm_n^1}) = U_n$
    and it is easy to see that $A U_n = U_n$.
\end{proof}

For any ultrametric $\delta \in U_n$, the point $\delta' = \delta + a(1,\ldots,1)^T$ is also an ultrametric for any real number $a$ since $\delta_{uv} + a \leq \max\{\delta_{uw}+a,\delta_{vw}+a\}$ for all triples $u,v,w \in [n]$.  Therefore $\trop(\cm^1_n)$ can be considerd as a subset of tropical projective space $\TP^{\binom{[n]}{2}-1}$ defined as the quotient $\rr^{\binom{[n]}{2}}/\rr(1,\ldots,1)^T$.

Ultrametrics on $[n]$ can be classified by their topology $T$.
Let $T$ be a rooted tree with leaves labeled by $[n]$.
A \emph{clade} of $T$ is the set of leaves below a given internal vertex.
A \emph{descendant} of an internal vertex $v$ of $T$ is a vertex $u$ in $T$
such that the unique path from $u$ to the root of $T$ contains $v$.
The trivial clade is $[n]$, the set of all leaves.
Let $\clade(T)$ denote the set of clades of $T$ and $\clade^\circ(T)$ the set of clades excluding the trivial clade.
Each rooted tree $T$ is completely determined by $\clade^\circ(T)$
(one can build a tree given its clades by first adding an internal node above
all the leaves in each minimal clade, then treating each minimal clade as a single leaf and proceeding inductively).
As a shorthand for a nonempty subset $\{i_1,\ldots,i_r\} \subseteq [n]$, we will also write $i_1\cdots i_r$.

\begin{ex}
    Let $T_1$ and $T_2$ be the trees in Figure \ref{fig:treeGraph}.
    Then $\clade^\circ(T_1) = \{12,123\}$ and $\clade^\circ(T_2) = \{13,24\}$.
\end{ex}

Let $K_T$ denote the closed cone consisting of all ultrametrics with topology $T$.
Like $U_n$, it has lineality space spanned by $(1,\ldots,1)^T$, so it can be considered as a subset of tropical projective space.

\begin{thm}[\cite{ardila2006bergman}, Proposition 3]\label{thm:simplicialComplex1}
The tropical Cayley-Menger variety of points in $\rr^1$, $\trop(\cm_1^n) = U_n \subseteq \TP^{\binom{[n]}{2}-1}$, admits a simplicial fan structure with cones $K_T$ for each rooted tree $T$ on leaves $[n]$, where $K_{T_1}$ is a face of $K_{T_2}$ if and only if $\clade(T_1) \subseteq \clade(T_2)$.
\end{thm}

We now introduce some notation for giving two different bases of the linear hull of $K_T$.

\begin{defn}\label{defn:basesOfLinearHull}
    For each $C \in \clade(T)$,
    we define two vectors, $v_C$ and $m_T^c$, in $\rr^{\binom{[n]}{2}}$ as follows.
    Let $v_C \in \rr^{\binom{[n]}{2}}$ be the characteristic vector of $\binom{C}{2}$.
    Let $m_T^C \in \rr^{\binom{[n]}{2}}$ be the characteristic vector of the set of pairs $ij$ in $[n]$ such that $C$ is the smallest clade containing $ij$.
    Let $M_T$ be the matrix with columns $m_T^C$.
    See Figure \ref{fig:ultrametricExample} for an example.
\end{defn}

A given ultrametric $\delta$ with topology $T$ can be expressed as
    \[ \delta = \sum_{C \in clade(T)} \delta_C m_T^C \]
where $\delta_C$ is the label assigned to clade $C$.
Thus, the columns of $M_T$ are a basis of the linear hull of $K_T$.
Within its linear span,
the cone $K_T$ is cut out by the set of inequalities
$\{\delta_{C} \leq \delta_{C'} | C' \in \clade(T), C \subseteq C'\}$.
The following lemma implies that another basis for the linear hull of $K_T$ is the set $\{v_C\}_{C \in \clade(T)}$.

\begin{lemma}\label{lemma:ultrametricConeGens}
	The cone $K_T$ of $U_n \subseteq \TP^{\binom{[n]}{2}-1}$ containing all ultrametrics with a given topology $T$ is generated by $\{-v_C\}_{C \in \clade^\circ(T)}$
    (modulo lineality space).
\end{lemma}
\begin{proof}
	Let $\delta$ be an ultrametric with topology $T$ and let $m = \max(\delta)$.
    Let $\delta'$ be the ultrametric obtained by labeling all internal vertices of $T$ by $m$.  So $\delta' = m(1,\ldots,1)^T = mv_{[n]}$.
    The ultrametric $\delta'$ can be turned into $\delta$ by iteratively
    decreasing the labels on each internal vertex and all its descendants.
    This corresponds to subtracting vectors of the form $v_C$.
    Concretely,
    \[ \delta = mv_{[n]} + \sum_{C \in clade^\circ(T)} -t_C v_C \]
    where $t_C = \delta_{C'} - \delta_{C}$ for all $C \in \clade^\circ(T)$ where $C'$ is the parent of $C$. 
    The condition $\delta_{C} \le \delta_{C'}$ gives $t_C \geq 0$, so $K_T$ consists of all nonnegative combinations of $\{-v_C\}_{C \in \clade^\circ(T)}$.
\end{proof}

\section{The tropical Pollaczek-Geiringer variety}\label{sec:tropicalLaman}

\begin{thm}\label{thm:tropicalLamanIsTrees}
    The tropicalization of the Pollaczek-Geiringer variety is the Minkowski sum of
    two copies of the set of ultrametrics on $[n]$.
    That is, $\trop(\cm_n^2) = U_n + U_n$.
\end{thm}
\begin{proof}
    As noted in \cite{capco2018number}, the usual parameterization of $\cm^2_n$
    given by $\delta_{uv} = (x_u-x_v)^2 + (y_u - y_v)^2$
    becomes $\delta_{uv} = (x_u-x_v)(y_u-y_v)$ after applying the following change of variables
    \[
        x_u \mapsto x_u + i y_u \qquad y_u \mapsto x_u - i y_u.
    \]
    Now let $\alpha: \cc^{\binom{[n]}{2}} \times \cc^{\binom{[n]}{2}} \rightarrow \cc^{\binom{[n]}{2}}$
    be the monomial map sending $(\delta^1_{uv},\delta^2_{wx})_{uv,wx}$ to $(\delta^1_{uv}\delta^2_{uv})_{uv}$.
    Under this new parameterization, it is clear that $\cm_n^2 = \alpha(\sqrt{\cm_n^1} \times \sqrt{\cm_n^1})$.
    The rows of the integer matrix $A$ representing $\alpha$ are $\{e_{uv} + f_{uv}\}_{uv}$ 
    where $\{e_{uv}\}_{uv},\{f_{uv}\}_{uv}$ are the canonical bases of each copy of $\cc^{\binom{[n]}{2}}$.
    Theorems \ref{thm:monomialMapTropicalVariety} and \ref{thm:ardilaKlivansBergmanFan} then imply the proposition.
\end{proof}

\begin{rmk}
 Proposition \ref{prop:CM1} and Thoerem \ref{thm:tropicalLamanIsTrees} describe $\trop(\cm_n^d)$ for $d = 1$ and $2$, and suggest a pattern that perhaps $\trop(\cm_n^d)$ might be equal to the sum of $d$ copies of $U_n$ for general $d$.  However we were not able to make such a generalization for $d \geq 3$.  The key observation in the $d = 2$ case is the factorization of the Euclidean distance $\delta_{uv} = (x_u-x_v)^2 + (y_u - y_v)^2$ into a product of a term involving only $x$-distance and a term involving only $y$-distnace.  We could not find an analogous factorization for $\delta_{uv} = (x_u-x_v)^2 + (y_u - y_v)^2 + (z_u - z_v)^2$, the Euclidean distance in $\rr^3$.
\end{rmk}

Our goal for the rest of this section is to prove Theorem \ref{thm:simplicialComplex} which describes a polyhedral fan structure on $\trop(\cm_2^n)$.

\begin{defn}\label{defn:tpn}
    The \emph{tree pair complex on $n$ leaves}, denoted $\tp(n)$, is the abstract simplicial complex on ground set $2^{[n]}$
    whose faces are all subsets of the form $\clade(T_1) \cup \clade(T_2)$ where $T_1$ and $T_2$
    are rooted trees on leaf set $[n]$.
\end{defn}

Note that any subset of $\clade(T)$ can be realized as $\clade(T')$ where
$T'$ is obtained from $T$ by contracting internal edges.
Thus the tree pair complex is indeed an abstract simplicial complex.
Definition \ref{defn:tpn} allows $T_1 = T_2$,
so $\tp(n)$ contains the simplicial complex implicit in Theorem \ref{thm:simplicialComplex1}
as a sub-complex.
We now state, but do not yet prove, our main theorem.

\begin{thm}\label{thm:simplicialComplex}
    The tropical Pollaczek-Geiringer variety $\trop(\cm_2^n)$ admits a simplicial fan structure isomorphic to $\tp(n)$.
\end{thm}


\begin{defn}
    Let $T_1,T_2$ be rooted trees on leaf set $[n]$.
    The \emph{clade graph} of $T_1$ and $T_2$ is the bipartite graph $G_{T_1,T_2} = (V_1,V_2,E)$
    whose partite vertex set $V_i$ is the set of clades of $T_i$
    and whose edge set $E = \{e_{ij} : 1 \le i < j \le n\}$
    has $e_{ij}$ connecting the minimal clades of $T_1$ and $T_2$ that contain
    both leaves $i$ and $j$.
\end{defn}

Proposition \ref{prop:cladeGraphGivesDimension} uses clade graphs
to derive the dimension of the cone $K_{T_1} + K_{T_2}$ from the combinatorics of $T_1$ and $T_2$.
Subgraphs of clade graphs will play a crucial role in our tropical proof of Laman's theorem
in the next section.

\begin{ex}\label{ex:treeConstruction}
    Figure \ref{fig:treeGraph} shows two rooted trees on vertex set $\{1,2,3,4\}$ alongside their clade graph.
    In both trees, the trivial clade $1234$ is the minimal clade containing the leaf pairs $14$ and $34$
    and so there is a double edge between both copies of the trivial clade.
    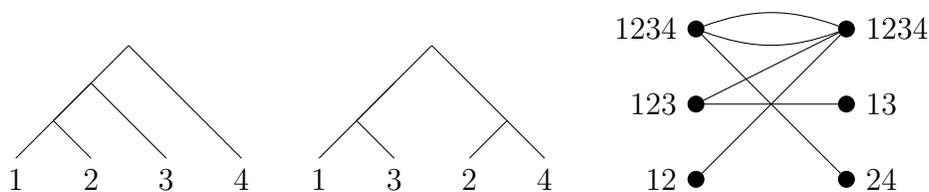
\begin{figure}[h!]
        \begin{tikzpicture}
            \node [below] (a) at (0,0){$1$};
            \node [below] (b) at (1,0){$2$};
            \node [below] (c) at (2,0){$3$};
            \node [below] (d) at (3,0){$4$};
            \draw (0,0) -- (1/2,1/2);
            \draw (1,0) -- (1/2,1/2);
            \draw (1/2,1/2) -- (1,1);
            \draw (2,0) -- (1,1);
            \draw (3,0) -- (3/2,3/2);
            \draw (1/2,1/2) -- (3/2,3/2);
        \end{tikzpicture}
        \quad
        \begin{tikzpicture}
            \node [below] (a) at (0,0){$1$};
            \node [below] (b) at (1,0){$3$};
            \node [below] (c) at (2,0){$2$};
            \node [below] (d) at (3,0){$4$};
            \draw (0,0) -- (1/2,1/2);
            \draw (1,0) -- (1/2,1/2);
            \draw (1/2,1/2) -- (1,1);
            \draw (2,0) -- (5/2,1/2);
            \draw (3,0) -- (3/2,3/2);
            \draw (1/2,1/2) -- (3/2,3/2);
        \end{tikzpicture}
        \quad
        \begin{tikzpicture}
            \vertex (12) at (0,0)[label=left:$12$]{};
            \vertex (123) at (0,1)[label=left:$123$]{};
            \vertex (1234) at (0,2)[label=left:$1234$]{};
            \vertex (14) at (2,0)[label=right:$24$]{};
            \vertex (23) at (2,1)[label=right:$13$]{};
            \vertex (1234') at (2,2)[label=right:$1234$]{};
            \path
                (12) edge (1234')
                (123) edge (1234')
                (123) edge (23)
                (1234) edge[bend right=20] (1234')
                (1234) edge[bend left=20] (1234')
                (1234) edge (14)
            ;
        \end{tikzpicture}
        \caption{Two rooted trees and their clade graph.}\label{fig:treeGraph}
    \end{figure}
\end{ex}

We now note that our tropical proof of Laman's theorem does not require any of the remaining results in this section.
Hence, the reader who is only interested in our tropical proof of Laman's theorem could skip to Section~\ref{sec:tropicalProofOfLaman} now.

\begin{prop}\label{prop:cladeGraphGivesDimension}
    For rooted trees $T_1,T_2$, the following values are equal:
    \begin{enumerate}
     \item\label{item:dimensionOfCone} the dimension of $K_{T_1} + K_{T_2}$,
     \item\label{item:rankOfGraphic} the rank of the graphic matroid of $G_{T_1,T_2}$ (the number of vertices minus the number of connected components),
     \item\label{item:cardinalityOfUnion} the cardinality of $\clade(T_1) \cup \clade(T_2)$.
    \end{enumerate}
\end{prop}
\begin{proof}
    Recall that the rank of the vertex-edge incidence matrix of a bipartite graph is equal to the rank of its graphic matroid.
    Equivalence of \eqref{item:dimensionOfCone} and \eqref{item:rankOfGraphic} then
    follows from the fact that $\begin{pmatrix} M_{T_1} & M_{T_2}\end{pmatrix}$ (see Definition \ref{defn:basesOfLinearHull})
    is the vertex-edge incidence matrix of $G_{T_1,T_2}$ and that its column span is the linear hull of $K_{T_1} + K_{T_2}$.

    Now we show that \eqref{item:dimensionOfCone} and \eqref{item:rankOfGraphic} are equivalent to \eqref{item:cardinalityOfUnion}.
    The linear hull of $K_{T_1} + K_{T_2}$ is also spanned by $\{v_C\}_{C \in \clade(T_1)} \cup \{v_C\}_{C \in \clade(T_2)}$, so
     \[ \dim(K_{T_1} + K_{T_2}) \leq |\clade(T_1) \cup \clade(T_2)|. \]
    We proceed by showing that $|\clade(T_1) \cup \clade(T_2)| \leq \rank(G_{T_1,T_2})$.
    Note that the number of vertices of $G_{T_1,T_2}$ is $|\clade(T_1)| + |\clade(T_2)|$
    so we must prove that the number of connected components of $G_{T_1,T_2}$ is at most $|\clade(T_1) \cap \clade(T_2)|$.
    
    We first show that that each clade $C$ of $T_1$ (without loss of generality) connects
    to the smallest clade of $T_2$ containing $C$.
    Let $\mathcal{D}$ be the set of clades in $T_2$ that are adjacent to $C$ in $G_{T_1,T_2}$.  Suppose $E_1,E_2 \in \mathcal{D}$ are disjoint and let $e_{ij}$ and $e_{kl}$ be the edges connecting them to $C$ respectively.  Among the set $\{i,j,k,l\}$ there are at least two other pairs besides $\{i,j\}$ and $\{k,l\}$ for which $C$ is the smallest clade in $T_1$ containing both.  Without loss of generality suppose $\{i,k\}$ is such a pair.  Then $e_{ik}$ connects $C$ to clade $E_3 \in \mathcal{D}$ that contains both $E_1$ and $E_2$.  It follows that $\mathcal{D}$ has a unique maximal element by inclusion, $D$.  For any $i \in C$, there exists $j \in C$ such that $e_{ij}$ is incident to $C$.  Therefore $i \in D$, so then $C \subseteq D$.  So $D$ must be the smallest clade of $T_2$ containing $C$. 

    If $C \in \clade(T_1) \cap \clade(T_2)$ then the two vertices in $G_{T_1,T_2}$ corresponding to $C$ are adjacent.  If $C \in \clade(T_1) \setminus \clade(T_2)$, then its vertex is adjacent to the vertex of a clade $D$ that strictly contains $C$.  Therefore there is a path from vertex $C$ through an ascending chain of clades that eventually reaches a shared clade.  Every vertex is connected to the vertex pair of a shared clade, so the number of connected components of $G_{T_1,T_2}$ is bounded by $|\clade(T_1) \cap \clade(T_2)|$.
\end{proof}

\begin{cor}\label{cor:maximalConesNoCommonClade}
    The pairs of trees $T_1,T_2$ for which $K_{T_1} + K_{T_2}$ has maximal dimension are those such that $T_1$ and $T_2$ are binary and have no nontrivial common clade.
\end{cor}

\begin{cor}\label{cor:vDescription}
  For any pair of trees $T_1,T_2$, the cone $K_{T_1} + K_{T_2} \subseteq \TP^{\binom{[n]}{2}-1}$ is a simplicial cone generated by $\{-v_C\}_{C \in \clade^\circ(T_1) \cup \clade^\circ(T_2)}$.
\end{cor}
\begin{proof}
  By Lemma \ref{lemma:ultrametricConeGens}, $K_{T_i}$ has lineality space $(1,\ldots,1)^T$ and is generated by $\{-v_C\}_{C \in \clade^\circ(T_i)}$ in $\TP^{\binom{[n]}{2}-1}$.  Therefore $K_{T_1} + K_{T_2}$ has $(1,\ldots,1)^T$ in its lineality space and is generated by $\{-v_C\}_{C \in \clade^\circ(T_1) \cup \clade^\circ(T_2)}$ in $\TP^{\binom{[n]}{2}-1}$.  By Proposition \ref{prop:cladeGraphGivesDimension},
  \[ \dim(K_{T_1} + K_{T_2}) = |\clade(T_1) \cup \clade(T_2)| = |\clade^\circ(T_1) \cup \clade^\circ(T_2)|+1. \]
  Modulo $(1,\ldots,1)^T$, the dimension of the cone is equal to the number of generators, so it must be simplicial.
\end{proof}

Corollary \ref{cor:vDescription} describes the cone $K_{T_1} + K_{T_2}$ in terms of its rays, i.e. a v-description.  This descirption implies that the cone depends only on $\clade(T_1) \cup \clade(T_2)$, and not any other properties of the trees.  For $\cals \subseteq 2^{[n]}$, let $\cals^\circ$ denote $\cals \setminus\{[n]\}$ and let $K_\cals \subseteq \TP^{\binom{[n]}{2}-1}$ denote the cone generated by $\{-v_C\}_{C\in \cals^\circ}$ (with lineality space $(1,\ldots,1)^T$).  Therefore $K_{T_1} + K_{T_2} = K_\cals$ for $\cals = \clade^\circ(T_1) \cup \clade^\circ(T_2)$.

In addition to a v-description of cone $K_\cals$, we would like an h-description: a system of linear equations and inequalties that cut out the cone.  This result is given in Proposition \ref{prop:hDescription}.  From the h-description we can say how the cones in $\trop(\cm^n_2)$ intersect, which will complete the proof of Theorm \ref{thm:simplicialComplex}.

Suppose $\cals \in \tp(n)$, so that $\cals = \clade(T_1) \cup \clade(T_2)$ for a pair of trees $T_1,T_2$, and $K_\cals = K_{T_1} + K_{T_2}$.
The \emph{clade intersection poset of $\cals$},
denoted $\cip(\cals)$,
will consist of $\cals$ and all intersections of elements of $\cals$
containing two or more elements,
and be partially ordered by inclusion.
In the h-description of $K_\cals$, there will be one equation or inequality for each element of $\cip(\cals)^\circ$ plus some additional equations coming from pairs of leaves within the same elements of $\cip(\cals)$, as we show below.
This construction guarantees that for any pair $ij \subseteq [n]$, there is a unique smallest set $C \in \cip(\cals)$ that contains $ij$.  Denote this set $\overline{ij}$.
Given $C \in \cip(\cals)$, the \emph{parents of $C$} are the elements of $\cip(\cals)$ that cover $C$,
and the \emph{children of $C$} are the elements of $\cip(\cals)$ that $C$ covers
(recall that $a$ is said to cover $b$ in a poset if $a$ is greater than $b$,
and there is no element strictly between $a$ and $b$).
We claim that $\cip(\cals)$ is a join-semilattice.
Otherwise, if the join of $A$ and $B$ does not exist, then there exist mutually incomparable $C_1,C_2 \in \cip(\cals)$
that are both minimal elements of $\cip(\cals)$ containing $A \cup B$.
But this is a contradiction since then $C_1\cap C_2 \in \cip(\cals)$ also contains $A \cup B$.
The join of $A$ and $B$ will be denoted $A \vee B$.

\begin{lemma}\label{lem:minSet}
 For $\cals \in \tp(n)$, and $C \subseteq [n]$ with $|C| \geq 2$, let $D$ be the minimal element of $\cip(\cals)$ that contains $C$.
 Then there exists a pair $i,j \in C$ such that $\overline{ij} = D$.
\end{lemma}
\begin{proof}
Suppose no such pair $ij$ exists, so every pair in $C$ appears in some child of $D$.
We claim that there exist three children $E_1,E_2,E_3$ of $D$ that have nontrivial pair-wise intersection.  Let $E_1$ be a child of $D$ that has maximal intersection with $C$ among the children of $D$ and let $a \in E_1 \cap C$.  Since $E_1$ does not contain $C$, there is some $b \in C \setminus E_1$.  Let $E_2$ be a child of $D$ that contains the pair $ab$.  By how $E_1$ was chosen, $E_2$ does not contain $E_1 \cap C$, so there is
$c \in (E_1 \cap C) \setminus E_2$ and $E_1$ contains $ac$.
Finally take $E_3$ to be a child of $D$ that contains $bc$.

We note that for any three sets in $\cals$, at least two of the sets
must be clades in the same tree,
implying that either one contains the other, or they are disjoint.
Therefore there cannot be three sets in $\cals$ with nontrivial pair-wise intersection and none containing another.
This implies that any element of $\cip(\cals)\setminus \cals$ is the intersection of exactly two elements of $\cals$.

Now, for each $k=1,2,3$, if $E_k \notin \cals$ then it is the intersection of $D$ and one other set $E'_k \in \cals$.
If $E_k \in \cals$ then let $E_k' = E_k$.
The sets $E_1',E_2',E_3'$ are all in $\cals$ and have nontrivial pair-wise intersection.
If $E_1'$ contains $E'_2$, then $E_2$ is a descendant of both $E_1'$ and $D$.
This implies $E_2 \subseteq E_1' \cap D = E_1$, which is a contradiction since $E_1$ and $E_2$ are both children of $D$.
Therefore the sets $E_1',E_2',E_3'$ do not satisfy any containment relations with each other
and no two are disjoint.
But we have already seen that this cannot happen.
\end{proof}

Modulo lineality space, any point $\delta \in K_\cals$ can be written uniquely as
  \[ \delta = \sum_{C \in \cals^\circ} -t_C v_C \]
with each $t_C \in \rr_{\geq 0}$.  Therefore the $ij$ coordinate has the form
  \[ \delta_{ij} = \sum_{\substack{C \in \cals^\circ \\ C \supseteq \overline{ij}}} -t_{C}. \]
It follows that if $\overline{ij} = \overline{kl}$ then
\begin{equation}\label{eqn:ijkl}
 \delta_{ij} = \delta_{kl}.
\end{equation}
With this in mind, we will write $\delta_C$ to denote some $\delta_{ij}$ with $C = \overline{ij}$.  By Lemma \ref{lem:minSet}, for every $C \in \cip(\cals)$ there is some $ij$ with $\overline{ij} = C$, so $\delta_C$ is well-defined.

We can also express each $t_C$ in terms of $d$.
For $C \in \cip(\cals)^\circ$,
let $D_1,\ldots,D_r$ be the parents of $C$.
For $I \subseteq [r]$,
let $D_I = \bigvee_{i\in I} D_i$ for $I \neq \emptyset$ and $D_{\emptyset} = C$.
Inclusion-exclusion gives
  \begin{equation}\label{eqn:tind}
   -\sum_{I \subseteq [r]} (-1)^{|I|} \delta_{D_I} = \sum_{I \subseteq [r]} (-1)^{|I|}\sum_{\substack{E \in \cals^\circ \\ E \supseteq D_I}} t_E =
   \begin{cases} t_C & \text{ if } C \in \cals^\circ \\ 0 & \text{ otherwise}\end{cases}.
  \end{equation}
For $C \in \cals^\circ$,
rewriting the known inequality $t_C \geq 0$ in terms of $\delta$ gives
 \begin{equation}\label{eqn:facets}
  \sum_{I \subseteq [r]} (-1)^{|I|} \delta_{D_I} \leq 0.
 \end{equation}
Statement \eqref{eqn:tind} also gives an equation on $\delta$ for each $C \in \cip(\cals) \setminus\cals$ with parents $D_1,\ldots,D_r$,
 \begin{equation}\label{eqn:cycle}
  \sum_{I \subseteq [r]} (-1)^{|I|} \delta_{D_I} = 0.
 \end{equation}
Let $F_\cals$ denote the system of inequalties and equations on $\delta$ from lines \eqref{eqn:ijkl},\eqref{eqn:facets},\eqref{eqn:cycle}.  We will prove in Proposition \ref{prop:hDescription} that $F_\cals$ is sufficient to cut out $K_\cals$, but first an example.

\begin{ex}
    We will construct the system $F_\cals$ when $\cals = \clade(T_1) \cup \clade(T_2)$ where $T_1$ and $T_2$ are the rooted trees shown below.
    (i.e. $\cals = \{12,123,56,456,14,134,26,256,123456\}$)
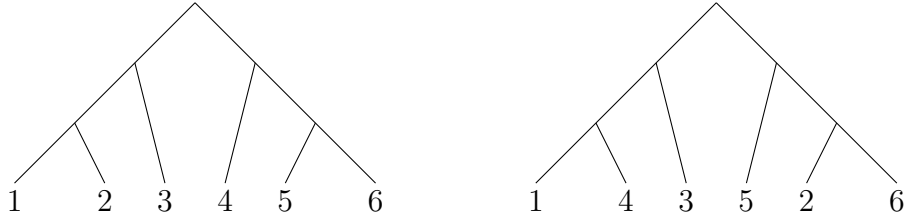
    \[
        \begin{tikzpicture}[scale=0.8]
            \draw(0,0)--(3,3);
            \draw(3,3)--(6,0);
            \draw(2,2)--(2.5,0);
            \draw(4,2)--(3.5,0);
            \draw(1,1)--(1.5,0);
            \draw(5,1)--(4.5,0);
            \node at (0,-0.3){$1$};
            \node at (1.5,-0.3){$2$};
            \node at (2.5,-0.3){$3$};
            \node at (3.5,-0.3){$4$};
            \node at (4.5,-0.3){$5$};
            \node at (6,-0.3){$6$};
        \end{tikzpicture}
        \qquad\qquad
        \begin{tikzpicture}[scale=0.8]
            \draw(0,0)--(3,3);
            \draw(3,3)--(6,0);
            \draw(2,2)--(2.5,0);
            \draw(4,2)--(3.5,0);
            \draw(1,1)--(1.5,0);
            \draw(5,1)--(4.5,0);
            \node at (0,-0.3){$1$};
            \node at (1.5,-0.3){$4$};
            \node at (2.5,-0.3){$3$};
            \node at (3.5,-0.3){$5$};
            \node at (4.5,-0.3){$2$};
            \node at (6,-0.3){$6$};
        \end{tikzpicture}.
    \]
    Then $\cip(\cals) = \cals \cup \{13\}$ since $13 = 123 \cap 134$
    and no other non-singleton non-empty sets arise as intersections of elements in $\cals$.
    The Hasse diagram of $\cip(\cals)$ is as follows
    
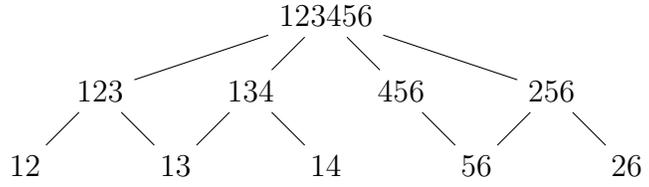
\[
        \begin{tikzpicture}[scale=1]
            \node (12) at (0,0){$12$};
            \node (13) at (2,0){$13$};
            \node (14) at (4,0){$14$};
            \node (56) at (6,0){$56$};
            \node (26) at (8,0){$26$};
            \node (123) at (1,1){$123$};
            \node (134) at (3,1){$134$};
            \node (456) at (5,1){$456$};
            \node (256) at (7,1){$256$};
            \node (123456) at (4,2){$123456$};
            \draw (12)--(123);
            \draw (13)--(123);
            \draw (14)--(134);
            \draw (13)--(134);
            \draw (56)--(456);
            \draw (56)--(256);
            \draw (26)--(256);
            \draw (123)--(123456);
            \draw (134)--(123456);
            \draw (456)--(123456);
            \draw (256)--(123456);
        \end{tikzpicture}.
    \]
    Considering $C \in \cals$ with $|C| > 2$, we have
    \begin{align*}
        \delta_{123456} :=& \ \delta_{15} = \delta_{16} = \delta_{24} = \delta_{35} = \delta_{36} \qquad \delta_{456} := \delta_{45} = \delta_{46}
        \\ &\delta_{123} := \delta_{23} \qquad \delta_{134} := \delta_{34} \qquad \delta_{256} := \delta_{25}.
    \end{align*}
    One more equality comes from $13 \in \cip(\cals) \setminus \cals$, namely
    \[
        \delta_{13} - \delta_{123} - \delta_{134} + \delta_{123456} = 0.
    \]
    Finally, we have the inequalities
    \begin{align*}
        &\delta_{12} \le \delta_{123} \qquad \delta_{14} \le \delta_{134} \qquad \delta_{56} -\delta_{456} - \delta_{256} + \delta_{123456} \le 0 \qquad \delta_{26} \le \delta_{256} \\
        \delta_{123} &\le \delta_{123456} \qquad \delta_{134} \le \delta_{123456} \qquad \delta_{456} \le \delta_{123456} \qquad \delta_{256} \le \delta_{123456}.
    \end{align*}
\end{ex}

\begin{prop}\label{prop:hDescription}
    For $\cals \in \tp(n)$, the polyhedral cone defined by the system $F_\cals$ is $K_\cals$.
\end{prop}
\begin{proof}
It has already been observed that $\delta \in K_\cals$ satisfies the system $F_\cals$.
The inequalities in $F_\cals$ are facet-defining, and define all facets of $K_\cals$,
because for each given $C \in \cals^\circ$,
\eqref{eqn:facets} achieves equality at all extreme rays of $K_\cals$ aside from $v_C$.

It remains to show that the linear space defined by the equations of $F_\cals$ is the linear hull of $K_\cals$.  For each pair $ij$, $\delta_{ij} = \delta_C$ for some $C \in \cip(\cals)$.  If $C \notin \cals$ then $\delta_C$ can be rewritten as a sum and difference of $\{\delta_D\}_{D \supsetneq C}$ using the equality in $F_\cals$ associated to $C$.  Since the maximal element of $\cip(\cals)$ is $[n] \in \cals$, by induction $\delta_C$ can be written in terms of $\{\delta_D\}_{D \in \cals,\; D \supseteq C}$.  Therefore the linear space defined by $F_\cals$ is parameterized by $\{\delta_C\}_{C \in \cals}$ so it has dimension at most $|\cals|$ including the lineality space.  We know this linear space contains $K_\cals$, which also has dimension $|\cals|$ by Proposition \ref{prop:cladeGraphGivesDimension}, so it must be equal to the linear hull of $K_\cals$.
\end{proof}

\begin{prop}\label{prop:intersectionsAsItShould}
 For $\cals,\cals' \in \tp(n)$,
  \[ K_\cals \cap K_{\cals'} = K_{\cals \cap \cals'}. \]
\end{prop}
\begin{proof}
 The generators of $K_{\cals \cap \cals'}$ in $\TP^{\binom{[n]}{2}-1}$ are the intersection of the generators of $K_\cals$ and $K_\cals'$.  This implies that $K_{\cals \cap \cals'} \subseteq K_\cals \cap K_{\cals'}$.

 To show that $K_{\cals} \cap K_{\cals'} \subseteq K_{\cals \cap \cals'}$, we work by induction on $m = |\cals \setminus \cals'|$.  For $m = 0$, $\cals \subseteq \cals'$ and $K_\cals \subseteq K_{\cals'}$, so the result follows.  For $m > 0$ assume the statement is true for all smaller values of $m$ and then choose $C \in \cals \setminus \cals'$.  Let $D$ be the smallest element of $\cip(\cals')$ such that $C \subseteq D$.
 
First suppose that $C \subsetneq D$.
By Lemma \ref{lem:minSet} there exists a pair $ij \subseteq C$ such that $D$ is the smallest element of $\cip(\cals')$ containing $ij$.
Fix $k \in D \setminus C$.
If $\delta \in K_S$ with $\delta = \sum_{E \in \cals^\circ} -t_E v_E$ such that $t_C > 0$,
then $\delta_{ik} > \delta_C \geq \delta_{ij}$.
However if $\delta \in K_{\cals'}$, then $\delta_{ik} \leq \delta_D = \delta_{ij}$.
Therefore $\delta \in K_\cals \cap K_{\cals'}$ has $t_C = 0$, so $K_\cals \cap K_{\cals'}$ is contained in a facet of $K_\cals$.
 
If $D = C$ then $D \in \cip(\cals') \setminus \cals'$.
Let $E_1,\ldots,E_r$ be the parents of $D$ in $\cip(S')$.
For a point $\delta \in K_{\cals'}$, $\delta_D = \delta_{E_\emptyset}$ satisfies
  \[ \sum_{I \subseteq [r]} (-1)^{|I|} \delta_{E_I} = 0. \]
 Each $E_I = \overline{i_Ij_I}$ in $\cip(\cals')$ for some pair $i_I,j_I$.  Note that for $I \neq \emptyset$, the pair $i_Ij_I$ is not contained in $C$.  For $\delta \in K_\cals$, $\delta_{i_\emptyset j_\emptyset}$ depends on the value of parameter $t_C$ since $i_\emptyset j_\emptyset$ is contained in $C$, while every other $\delta_{i_Ij_I}$ does not.  Therefore if $\delta \in K_\cals$ is generic, the equation
  \[ \delta_{ij} = -\sum_{\substack{I \subseteq [r]\\ I \neq \emptyset}} (-1)^{|I|} \delta_{i_Ij_I} \]
 is not satisfied.  Therefore $K_\cals \cap K_{\cals'}$ has strictly lower dimension than $K_\cals$, so it must be contained in a facet of $K_\cals$.
 
In either case let $K_{\cals''}$ be the facet of $K_\cals$ containing $K_\cals \cap K_{\cals'}$ so that $\cals \cap \cals' \subseteq \cals'' \subsetneq \cals$.
Since $|\cals'' \setminus \cals'| < m$, by the induction hypothesis,
\[ K_{\cals} \cap K_{\cals'} = K_{\cals''} \cap K_{\cals'} \subseteq K_{\cals'' \cap \cals'} = K_{\cals \cap \cals'}. \qedhere\]
\end{proof}

Theorem \ref{thm:simplicialComplex} follows from Proposition \ref{prop:intersectionsAsItShould} and Corollary \ref{cor:vDescription}
by sending $\cals \in \tp(n)$ to $K_\cals$.

\section{A tropical proof of Laman's Theorem}\label{sec:tropicalProofOfLaman}
We now give our tropical proof of Laman's theorem.
Lemma~\ref{lemma:tropicalPreservesAlgebraicMatroid} allows us to determine the algebraic matroid underlying $\cm_n^2$
via projections of the tropicalization of $\cm_n^2$.
Theorem~\ref{thm:tropicalLamanIsTrees} allows us to translate geometric properties of this tropical variety
into combinatorial statements about pairs of rooted trees.

Given a graph $H$, let $V(H)$ and $E(H)$ denote the vertex and edge sets of $H$.
Each graph $H$ on vertex set $[n]$ describes a coordinate projection $\pi_H:\rr^{\binom{[n]}{2}} \to \rr^{E(H)}$.
Moreover, Lemma \ref{lemma:indSetCM}(3) and Lemma~\ref{lemma:tropicalPreservesAlgebraicMatroid}
imply that $H$ is generically rigid in $\rr^d$ if and only if $\pi_H(\trop(\cm_n^d))$ has the maximal dimension, $dn - \binom{d+1}{2}$.
For a tree $T$, define the matrix $M_T^H$ to be the submatrix of $M_T$ obtained by taking only the rows corresponding to $E(H)$.
The cone $\pi_H(K_T) \subseteq \pi_H(U_n)$ has linear hull equal to the span of $M_T^H$.
Define the {\em restricted clade graph} of $T_1$ and $T_2$ to be the subgraph
$G_{T_1,T_2}^H$ of $G_{T_1,T_2}$ on the same vertex set
whose edge set $E = \{e_{ij} : \{i,j\} \in E(H)\}$ has $e_{ij}$ connecting the minimal clades of $T_1$ and $T_2$ that contain $ij$.
For $S\subseteq [n]$, let $c_{T_i}(S)$ denote the smallest clade of $T_i$ containing $S$.
For each edge $ij \in E(H)$, note that $e_{ij}$ connects $c_{T_1}(ij)$ to $c_{T_2}(ij)$.
Now we give the analog of Proposition \ref{prop:cladeGraphGivesDimension} for coordinate projections.


\begin{prop}\label{prop:tropicalMinRigidCondition}
    A graph $H$ is minimally generically rigid in $\rr^2$ if and only if there is a pair of rooted binary trees $T_1,T_2$ such that $G_{T_1,T_2}^H$ is a tree.
\end{prop}
\begin{proof}
    By Lemma \ref{lemma:tropicalPreservesAlgebraicMatroid}, it suffices to show that $|E(H)| = 2n-3$ and $\pi_H(\trop(\cm_n^2))$
    has dimension $2n-3$ if and only if there are rooted binary trees $T_1,T_2$ such that $G_{T_1,T_2}^H$ is a tree.
    A rooted binary tree has exactly $n-1$ clades, so $G_{T_1,T_2}^H$ has $2n-2$ vertices.  It is a tree if and only if it is connected and has $2n-3$ edges.
    
    The edge sets of $H$ and $G_{T_1,T_2}^H$ are in bijection, so one has size $2n-3$ if and only if the other does.
    The dimension of $\pi_H(\trop(\cm_n^2)$ will be $2n-3$ if and only if there exists a cone $K_{T_1} + K_{T_2}$ of $\trop(\cm_n^2)$
    such that $\pi_H(K_{T_1} + K_{T_2})$ has dimension $2n-3$.
    The linear hull of $\pi_H(K_{T_1} + K_{T_2})$ is the column span of the adjacency matrix of $G_{T_1,T_2}^H$
    and so the dimension $\pi_H(K_{T_1}+K_{T_2})$ is the rank of this adjacency matrix.
    The rank of the adjacency matrix of a bipartite graph is the number of vertices minus the number of connected components.
    Therefore the adjacency matrix of $G_{T_1,T_2}^H$ has rank $2n-3$ if and only if $G_{T_1,T_2}^H$ is connected.
\end{proof}
 

To reprove Laman's theorem, it remains to show that the graphs $H$ satisfying the condition of Proposition \ref{prop:tropicalMinRigidCondition} are precisely the Laman graphs.
We will do this via the \emph{Henneberg moves},
which were shown by Henneberg in 1911 \cite{henneberg1911graphische} to generate precisely
the graphs which are minimally generically rigid in the plane.
We now define two conditions a graph can satisfy, then use our combinatorial description of $\trop(\cm_n^2)$ to show they are
both equivalent to the property of being generically minimally rigid in the plane. 

\begin{defn}
    Let $H$ be a graph with vertex set $[n]$ and edge set $E$.
    We say that $H$ is
    \begin{itemize}
    \item \emph{Laman} if $H$ has $2n-3$ edges and every subgraph of $H$ with $v$ vertices has at most $2v-3$ edges,
        \item \emph{Henneberg} if $H$ is the complete graph $K_2$, or $H$ can be obtained from a smaller Henneberg graph by either of the two \emph{Henneberg moves},
        which are
            \begin{enumerate}
                \item\label{hen1} adding a new vertex adjacent to two existing vertices, and
                \item\label{hen2} removing an edge $ij$ and adding a new vertex that is adjacent to $i$ and $j$ and some other vertex.
            \end{enumerate}
    \end{itemize}
\end{defn}

\begin{lemma}\label{lem:LtoH}
    If $H$ is Laman, then $H$ is Henneberg.
\end{lemma}
\begin{proof}
    This is well-known (see e.g. \cite{graver2008combinatorial}) but we provide a proof anyway to keep our
    proof of Laman's theorem self-contained.
    So let $H$ be a Laman graph.  We work by induction on $n$.
    If $n=2$ then $H = K_2$ which is Henneberg, so assume $H$ has at least three vertices.
    It is easy to check that since $H$ is Laman, each vertex has degree at least $2$.
    Assume $H$ has a vertex $v$ of degree exactly $2$.
    Then $H\setminus \{v\}$ is Laman, and therefore Henneberg by the induction hypothesis.
    $H$ can be obtained from $H\setminus \{v\}$ by attaching $v$ via the first Henneberg move.

    Now assume the minimum degree of $H$ is at least $3$.
    Since $H$ has $2n-3$ edges, some vertex $v$ must have degree $3$.
    Denote the neighbors of $v$ by $1,2,3$.
    If $123v$ is a clique then $H$ is not Laman, so there must be at least one edge missing which we take to be $12$.
    Let $H'$ be the graph obtained from $H\setminus \{v\}$ by adding the edge $12$.
    If $H'$ it not Laman, it has a strict subgraph $H''$ with $k$ vertices and $2k-2$ edges that includes the edge $12$.
    But then $H$ would violate the Laman condition as well, since the graph obtained from $H''$ by
    removing the edge $12$ and connecting $v$ to $1,2,3$
    would be a subgraph of $H$ containing $k+1$ vertices and $2(k+1)-2$ edges.
    So by the induction hypothesis, $H'$ is Henneberg and $H$ can be obtained from $H'$ via the second Henneberg move.
\end{proof}

Given a rooted tree $T$ on leaf set $[n]$,
the {\em restriction} of $T$ to $S \subseteq [n]$ is the rooted tree $T'$ obtained from the induced subtree of $T$ with leaves $S$ and their ancestors,
contracting away degree 2 vertices.
If $d$ is an ultrametric with tree topology $T$, then the restriction of $T$ to $S$ is the topology of the restriction of $d$ to the coordinates $\binom{S}{2}$.
Now let $H$ be a graph on vertex set $[n]$,
let $T_1,T_2$ be rooted trees on leaf set $[n]$, and let $H'$ be a subgraph of $H$.
If $T'_1,T'_2$ are the restrictions of $T_1,T_2$ to $V(H')$, the natural inclusion map $\eta:H' \to H$ induces an injective graph homomorphism
 \[ \tilde{\eta}: G_{T'_1,T'_2}^{H'} \to G_{T_1,T_2}^H \]
by sending clade $C$ of $T'_i$ to $c_{T_i}(C)$.  To see that $\tilde{\eta}$ maps edges to edges, note that for each $jk \in E(H')$, $\tilde{\eta}(c_{T'_i}(jk)) = c_{T_i}(jk)$, so the edge $e_{jk}$ of $G_{T'_1,T'_2}^{H'}$ goes to $e_{jk}$ of $G_{T_1,T_2}^{H}$.

\begin{ex}
    Let $H$ be the graph on vertex set $\{1,2,3,4\}$ pictured below
    and let $T_1$ and $T_2$ be as in Example \ref{ex:treeConstruction}.
    Let $H'$ be the subgraph of $H$ induced on vertex set $\{2,3,4\}$.
    Then $\eta$ is the inclusion of $H'$ in $H$,
    and $\tilde\eta$ maps each vertex labeled $234$ in $G_{T_1',T_2'}^{H'}$
    to the vertex on the corresponding side labeled $1234$ in $G_{T_1,T_2}^H$,
    maps the vertex labeled $23$ in $G_{T_1',T_2'}^{H'}$ to the vertex labeled $123$ in $G_{T_1,T_2}^H$,
    and maps the vertex labeled $24$ in $G_{T_1',T_2'}^{H'}$ to the vertex with the same label in $G_{T_1,T_2}^H$.
    Note that $\tilde\eta$ is a graph homomorphism.
    
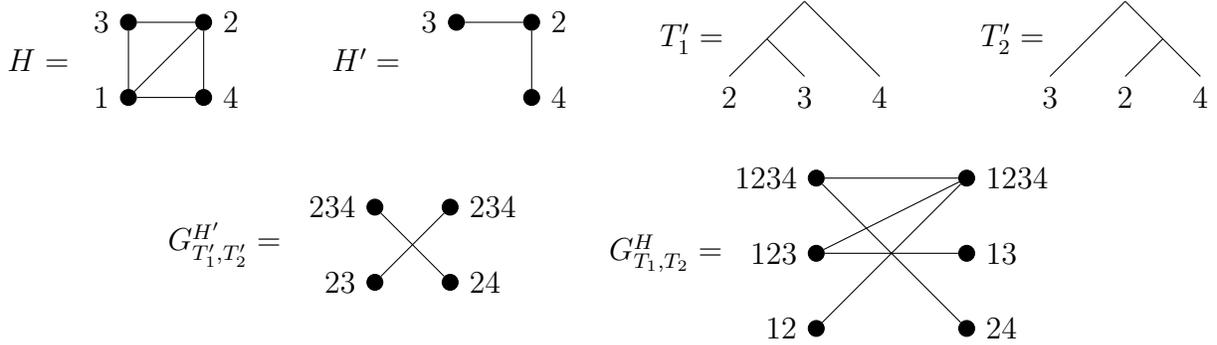
\begin{center}
        \begin{tikzpicture}
            \node at (-1.2,0.5){$H=$};
            \vertex (1) at (0,0)[label=left:$1$]{};
            \vertex (4) at (1,0)[label=right:$4$]{};
            \vertex (2) at (1,1)[label=right:$2$]{};
            \vertex (3) at (0,1)[label=left:$3$]{};
            \path
                (1) edge (2) edge (3) edge (4)
                (2) edge (3) edge (4)
            ;
        \end{tikzpicture}
        \qquad
        \begin{tikzpicture}
            \node at (-1.2,0.5){$H'=$};
            \vertex (4) at (1,0)[label=right:$4$]{};
            \vertex (2) at (1,1)[label=right:$2$]{};
            \vertex (3) at (0,1)[label=left:$3$]{};
            \path
                (2) edge (3) edge (4)
            ;
        \end{tikzpicture}
        \qquad
        \begin{tikzpicture}
            \node at (-.5,.5){$T_1'=$};
            \node [below] (b) at (0,0){$2$};
            \node [below] (c) at (1,0){$3$};
            \node [below] (d) at (2,0){$4$};
            \draw (0,0) -- (1/2,1/2);
            \draw (1,0) -- (1/2,1/2);
            \draw (1/2,1/2) -- (1,1);
            \draw (2,0) -- (1,1);
        \end{tikzpicture}
        \qquad
        \begin{tikzpicture}
            \node at (-.5,.5){$T_2'=$};
            \node [below] (b) at (0,0){$3$};
            \node [below] (c) at (1,0){$2$};
            \node [below] (d) at (2,0){$4$};
            \draw (0,0) -- (1,1);
            \draw (1,0) -- (3/2,1/2);
            \draw (2,0) -- (1,1);
        \end{tikzpicture}
        \\ \ \\
        \begin{tikzpicture}
            \node at (0,-.75){};
            \node at (-2,.5){$G^{H'}_{T_1',T_2'}=$};
            \vertex (23) at (0,0)[label=left:$23$]{};
            \vertex (234) at (0,1)[label=left:$234$]{};
            \vertex (24) at (1,0)[label=right:$24$]{};
            \vertex (234') at (1,1)[label=right:$234$]{};
            \path
                (23) edge (234')
                (234) edge (24)
            ;
        \end{tikzpicture}
        \qquad
        \begin{tikzpicture}
            \node at (-2,1){$G^H_{T_1,T_2}=$};
            \vertex (12) at (0,0)[label=left:$12$]{};
            \vertex (123) at (0,1)[label=left:$123$]{};
            \vertex (1234) at (0,2)[label=left:$1234$]{};
            \vertex (14) at (2,0)[label=right:$24$]{};
            \vertex (23) at (2,1)[label=right:$13$]{};
            \vertex (1234') at (2,2)[label=right:$1234$]{};
            \path
                (12) edge (1234')
                (123) edge (1234')
                (123) edge (23)
                (1234) edge (1234')
                (1234) edge (14)
            ;
        \end{tikzpicture}
    \end{center}
\end{ex}

\begin{lemma}\label{lem:TtoL}
 If $H$ is not a Laman graph, then $G_{T_1,T_2}^H$ is not a tree for any choice of pair of rooted trees $T_1,T_2$.
\end{lemma}
\begin{proof}
If $H$ has $n$ vertices but does not have $2n-3$ edges, then $H$ is not Laman and $G_{T_1,T_2}^H$ has the wrong number of edges to be a tree.
Suppose then that $H$ has $2n-3$ edges but is not Laman.
Then $H$ has a subgraph $H'$ with $n'$ vertices such that $H'$ has more than $2n'-3$ edges.
For any choice of trees $T_1,T_2$, let $T'_1,T'_2$ be the respective restrictions to $V(H')$.  Since $G_{T'_1,T'_2}^{H'}$ has $2n'-2$ vertices and more than $2n'-3$ edges, it must contain a cycle.  The graph homomorphism $\tilde{\eta}:G_{T'_1,T'_2}^{H'} \to G_{T_1,T_2}^{H}$ shows that $G_{T_1,T_2}^{H}$ must also contain a cycle.
\end{proof}

\begin{lemma}\label{lem:HtoT}
 If $H$ is Henneberg, then there exists a pair of rooted binary trees $T_1,T_2$ such that $G_{T_1,T_2}^H$ is a tree.
\end{lemma}
\begin{proof}
If $H$ is Hennberg with $n$ vertices, then it has $2n-3$ edges, and so $G_{T_1,T_2}^H$ also has $2n-3$ edges.  Then to prove that $G_{T_1,T_2}^H$ is tree, we show that it has $2n-2$ vertices and is connected.

We work by induction on $n$.
In the base case $n=2$, the only Henneberg graph is $H= K_2$.
Let $T_1 = T_2$ be the unique rooted binary tree on two leaves.
Then $G_{T_1,T_2}^H$ has two vertices, one for the clade $12$ in each tree, connected by edge $e_{12}$.
For $n > 2$, $H$ can be obtained (after relabeling) from a Henneberg graph $H'$ on $[n-1]$ by one of the Henneberg moves.
By the induction hypothesis, there are rooted binary trees $T'_1,T'_2$ such that $G_{T'_1,T'_2}^{H'}$ is connected.
 
First suppose that $H$ is obtained from $H'$
by a Henneberg move of type (\ref{hen1}) by adding vertex $n$ and connecting it to vertices 1 and 2 (without loss of generality). 
Let $T_1$ be the tree obtained from $T'_1$ attaching $n$ so that $1n$ becomes a clade.
Let $T_2$ be obtained from $T'_2$ by attaching $n$ so that $2n$ becomes a clade.
Since $H'$ is a subgraph of $H$ and $T'_i$ is the restriction of $T_i$ to $[n-1]$,
there is graph homomorphism $\tilde{\eta}:G_{T'_1,T'_2}^{H'} \to G_{T_1,T_2}^{H}$ defined as above.
Therefore $\tilde{\eta}(G_{T'_1,T'_2}^{H'})$ is connected.
$G_{T_1,T_2}^{H}$ has exactly two new clades not in the image of $\tilde{\eta}$, which are $1n$ in $T_1$ and $2n$ in $T_2$.
It also has two new edges, $e_{1n}$ connecting $1n$ to $c_{T_2}(1n) = \tilde{\eta}(c_{T'_2}(12))$ and
$e_{2n}$ connecting $2n$ to $c_{T_1}(2n) = \tilde{\eta}(c_{T'_1}(12))$.  Therefore the two new vertices are connected to $\tilde{\eta}(G_{T'_1,T'_2}^{H'})$, so $G_{T_1,T_2}^{H}$ is connected.

Now suppose that $H$ is obtained from $H'$ by a Henneberg move of type (\ref{hen2})
by removing edge $12$, adding vertex $n$ and adding edges $1n,2n,3n$ (without loss of generality).
We will construct $T_i$ from $T'_i$ by adding leaf $n$ and a new clade $C_i\cup\{n\}$ for some chosen clade or singleton set $C_i$ of $T'_i$.
Let $H''$ be $H'$ minus the edge $12$, so it is a subgraph of $H$, and
there is graph homomorphism $\tilde{\eta}:G_{T'_1,T'_2}^{H''} \to G_{T_1,T_2}^{H}$.
The graph $G_{T'_1,T'_2}^{H''}$ has two connected components with $c_{T'_1}(12)$ and $c_{T'_2}(12)$ in different ones.  Therefore $\tilde{\eta}(G_{T'_1,T'_2}^{H''})$ also has two connected components.  Thus we must choose $C_1$ and $C_2$ so that the edges $e_{1n},e_{2n},e_{3n}$ connect the two connected components of $\tilde{\eta}(G_{T'_1,T'_2}^{H''})$ and the two new clades $C_1\cup\{n\}$ and $C_2\cup\{n\}$.  The way $C_1$ and $C_2$ are chosen will depend on the relative positions of $1,2,3$ in $T'_1$ and $T'_2$.  We divide the situations into three cases, listed below and pictured in Figure \ref{fig:lamanCases}.
 
 {\bf Case 1}: Suppose 1 and 2 are closer to each other than to 3 in both $T'_1$ and $T'_2$ and that $c_{T'_1}(123)$ and $c_{T'_2}(123)$ are in different connected components of $G_{T'_1,T'_2}^{H''}$.  Let $C_1 = \{1\}$ and $C_2 = c_{T'_2}(12)$.
 \begin{itemize}
  \item $e_{1n}$ connects $C_1\cup\{n\}$ to $C_2\cup\{n\}$.
  \item $e_{2n}$ connects $\tilde{\eta}(c_{T'_1}(12))$ to $C_2\cup\{n\}$.
  \item $e_{3n}$ connects $\tilde{\eta}(c_{T'_1}(123))$ to $\tilde{\eta}(c_{T'_2}(123))$.
 \end{itemize}
 Therefore $G_{T_1,T_2}^{H}$ is connected.
 
 {\bf Case 2}: Suppose 1 and 2 are closer to each other than to 3 in both $T'_1$ and $T'_2$ and that $c_{T'_1}(123)$ and $c_{T'_2}(123)$ are in the same component of $G_{T'_1,T'_2}^{H''}$.  Either $c_{T'_1}(12)$ or $c_{T'_2}(12)$ are in the opposite component, so without loss of generality take it to be $c_{T'_1}(12)$.  Let $C_1 = \{1\}$ and $C_2 = \{3\}$.
 \begin{itemize}
  \item $e_{1n}$ connects $C_1\cup\{n\}$ to $\tilde{\eta}(c_{T'_2}(123))$.
  \item $e_{2n}$ connects $\tilde{\eta}(c_{T'_1}(12))$ to $\tilde{\eta}(c_{T'_2}(123))$.
  \item $e_{3n}$ connects $\tilde{\eta}(c_{T'_1}(123))$ to $C_2\cup\{n\}$.
 \end{itemize}
 Therefore $G_{T_1,T_2}^{H}$ is connected.
 
 {\bf Case 3}: Suppose 3 is closer to 1 or 2 than to the other in one of $T'_1$ or $T'_2$.  Without loss of generality, take 3 and 1 closer to each other than to 2 in $T'_2$.  Let $C_1 = \{1\}$ and $C_2 = c_{T'_2}(13)$.
 \begin{itemize}
  \item $e_{1n}$ connects $C_1\cup\{n\}$ to $C_2\cup\{n\}$.
  \item $e_{2n}$ connects $\tilde{\eta}(c_{T'_1}(12))$ to $\tilde{\eta}(c_{T'_2}(12))$.
  \item $e_{3n}$ connects $\tilde{\eta}(c_{T'_1}(13))$ to $C_2\cup\{n\}$.
 \end{itemize}
 Therefore $G_{T_1,T_2}^{H}$ is connected.
 \end{proof}
 
     \begin{figure}[!ht]
        \captionsetup{width=.9\linewidth}
        \caption{For $T_1$ and $T_2$ trees as in the proof of Lemma \ref{lem:HtoT}, on the left are the restrictions of $T_1$ and $T_2$ to the leaf set $123n$ and on the right is the restricted clade graph $G_{T_1,T_2}^H$.  Each label on $G_{T_1,T_2}^H$ indicates the clade that is the smallest containing those elements in $T_1$ and $T_2$ on the left and right respectively.  Solid outlines represent connected components of $G_{T'_1,T'_2}^{H''}$.  The leaf $n$ in each tree $T_1,T_2$ is placed such that the new edges $e_{1n},e_{2n},e_{3n}$ reconnect the graph, including the two new vertices, labeled $n$.}\label{fig:lamanCases}
        \begin{subfigure}[b]{1\textwidth}
        \centering
        \begin{tikzpicture}
            \node [below] (a) at (0,0){$1$};
            \node [below] (b) at (1,0){$n$};
            \node [below] (c) at (2,0){$2$};
            \node [below] (d) at (3,0){$3$};
            \node [above] (T1) at (3/2,3/2){$T_1$};
            \draw (0,0) -- (1/2,1/2);
            \draw (1,0) -- (1/2,1/2);
            \draw (1/2,1/2) -- (1,1);
            \draw (2,0) -- (1,1);
            \draw (3,0) -- (3/2,3/2);
            \draw (1/2,1/2) -- (3/2,3/2);
        \end{tikzpicture}
        \quad
        \begin{tikzpicture}
            \node [below] (a) at (0,0){$1$};
            \node [below] (b) at (1,0){$2$};
            \node [below] (c) at (2,0){$n$};
            \node [below] (d) at (3,0){$3$};
            \node [above] (T1) at (3/2,3/2){$T_2$};
            \draw (0,0) -- (1/2,1/2);
            \draw (1,0) -- (1/2,1/2);
            \draw (1/2,1/2) -- (1,1);
            \draw (2,0) -- (1,1);
            \draw (3,0) -- (3/2,3/2);
            \draw (1/2,1/2) -- (3/2,3/2);
        \end{tikzpicture}
        \quad
        \begin{tikzpicture}
            \vertex (11) at (0,0)[label=left:$n\;$]{};
            \vertex (12) at (0,1)[label=left:$12\;$]{};
            \vertex (13) at (0,2)[label=left:$123\;$]{};
            \vertex (21) at (2,0)[label=right:$\;n$]{};
            \vertex (22) at (2,1)[label=right:$\;12$]{};
            \vertex (23) at (2,2)[label=right:$\;123$]{};
            \path
                (11) edge (21)
                (12) edge (21)
                (13) edge (23)
                (12) edge[dashed] (22)
            ;
            \draw (0,1.5) ellipse (.4 and .8);
            \draw (2,1.5) ellipse (.4 and .8);
        \end{tikzpicture}
        \caption{Case 1: 1 and 2 are closest in both $T'_1$ and $T'_2$, and the vertices labeled $123$ are in opposite components.  The components detached by deleting $e_{12}$ are reconnected by $e_{3n}$.  The pairing of the top and middle vertices into components could also be reversed from what is pictured.}
        \end{subfigure}
        
        \begin{subfigure}[b]{1\textwidth}
        \centering
        \begin{tikzpicture}
            \node [below] (a) at (0,0){$1$};
            \node [below] (b) at (1,0){$n$};
            \node [below] (c) at (2,0){$2$};
            \node [below] (d) at (3,0){$3$};
            \node [above] (T1) at (3/2,3/2){$T_1$};
            \draw (0,0) -- (1/2,1/2);
            \draw (1,0) -- (1/2,1/2);
            \draw (1/2,1/2) -- (1,1);
            \draw (2,0) -- (1,1);
            \draw (3,0) -- (3/2,3/2);
            \draw (1/2,1/2) -- (3/2,3/2);
        \end{tikzpicture}
        \quad
        \begin{tikzpicture}
            \node [below] (a) at (0,0){$1$};
            \node [below] (b) at (1,0){$2$};
            \node [below] (c) at (2,0){$n$};
            \node [below] (d) at (3,0){$3$};
            \node [above] (T1) at (3/2,3/2){$T_2$};
            \draw (0,0) -- (1/2,1/2);
            \draw (1,0) -- (1/2,1/2);
            \draw (1/2,1/2) -- (1,1);
            \draw (2,0) -- (5/2,1/2);
            \draw (3,0) -- (3/2,3/2);
            \draw (1/2,1/2) -- (3/2,3/2);
        \end{tikzpicture}
        \quad
        \begin{tikzpicture}
            \vertex (11) at (0,0)[label=left:$n\;$]{};
            \vertex (12) at (0,1)[label=left:$12\;$]{};
            \vertex (13) at (0,2)[label=left:$123\;$]{};
            \vertex (21) at (2,0)[label=right:$\;n$]{};
            \vertex (22) at (2,1)[label=right:$\;12$]{};
            \vertex (23) at (2,2)[label=right:$\;123$]{};
            \path
                (11) edge (23)
                (12) edge (23)
                (13) edge (21)
                (12) edge[dashed] (22)
            ;
            \draw (-0.3,2) to [out=90,in=180] (1,2.5)
                           to [out=0,in=100] (2.3,2)
                           to [out=280,in=80] (2.3,1)
                           to [out=260,in=0] (2,0.6)
                           to [out=180,in=0] (0.7,1.5)
                           to [out=180,in=270] (-0.3,2);
            \draw (0,1) ellipse (.4 and .4);
        \end{tikzpicture}
        \caption{Case 2: 1 and 2 are closest in both $T'_1$ and $T'_2$, and the vertices labeled $123$ are in the same components.  The components detached by deleting $e_{12}$ are reconnected by $e_{2n}$.}
        \end{subfigure}
        
        \begin{subfigure}[b]{1\textwidth}
        \centering
        \begin{tikzpicture}
            \node [below] (a) at (0,0){$1$};
            \node [below] (b) at (1,0){$n$};
            \node [below] (c) at (2,0){$2$};
            \node [below] (d) at (3,0){$3$};
            \node [above] (T1) at (3/2,3/2){$T_1$};
            \draw (0,0) -- (1/2,1/2);
            \draw (1,0) -- (1/2,1/2);
            \draw (1/2,1/2) -- (1,1);
            \draw (2,0) -- (1,1);
            \draw (3,0) -- (3/2,3/2);
            \draw (1/2,1/2) -- (3/2,3/2);
        \end{tikzpicture}
        \quad
        \begin{tikzpicture}
            \node [below] (a) at (0,0){$1$};
            \node [below] (b) at (1,0){$3$};
            \node [below] (c) at (2,0){$n$};
            \node [below] (d) at (3,0){$2$};
            \node [above] (T1) at (3/2,3/2){$T_2$};
            \draw (0,0) -- (1/2,1/2);
            \draw (1,0) -- (1/2,1/2);
            \draw (1/2,1/2) -- (1,1);
            \draw (2,0) -- (1,1);
            \draw (3,0) -- (3/2,3/2);
            \draw (1/2,1/2) -- (3/2,3/2);
        \end{tikzpicture}
        \quad
        \begin{tikzpicture}
            \vertex (11) at (0,0)[label=left:$n\;$]{};
            \vertex (12) at (0,1)[label=left:$12\;$]{};
            \vertex (13) at (0,2)[label=left:$13\;$]{};
            \vertex (21) at (2,0)[label=right:$\;n$]{};
            \vertex (22) at (2,1)[label=right:$\;13$]{};
            \vertex (23) at (2,2)[label=right:$\;12$]{};
            \path
                (11) edge (21)
                (12) edge[bend left=20] (23)
                (13) edge (21)
                (12) edge[dashed, bend right=20] (23)
            ;
            \draw (1,1.5) ellipse (1.4 and 1.1) [dashed];
        \end{tikzpicture}\\ \vspace{2em}
        \begin{tikzpicture}
            \node [below] (a) at (0,0){$1$};
            \node [below] (b) at (1,0){$n$};
            \node [below] (c) at (2,0){$3$};
            \node [below] (d) at (3,0){$2$};
            \node [above] (T1) at (3/2,3/2){$T_1$};
            \draw (0,0) -- (1/2,1/2);
            \draw (1,0) -- (1/2,1/2);
            \draw (1/2,1/2) -- (1,1);
            \draw (2,0) -- (1,1);
            \draw (3,0) -- (3/2,3/2);
            \draw (1/2,1/2) -- (3/2,3/2);
        \end{tikzpicture}
        \quad
        \begin{tikzpicture}
            \node [below] (a) at (0,0){$1$};
            \node [below] (b) at (1,0){$3$};
            \node [below] (c) at (2,0){$n$};
            \node [below] (d) at (3,0){$2$};
            \node [above] (T1) at (3/2,3/2){$T_2$};
            \draw (0,0) -- (1/2,1/2);
            \draw (1,0) -- (1/2,1/2);
            \draw (1/2,1/2) -- (1,1);
            \draw (2,0) -- (1,1);
            \draw (3,0) -- (3/2,3/2);
            \draw (1/2,1/2) -- (3/2,3/2);
        \end{tikzpicture}
        \quad
        \begin{tikzpicture}
            \vertex (11) at (0,0)[label=left:$n\;$]{};
            \vertex (12) at (0,1)[label=left:$13\;$]{};
            \vertex (13) at (0,2)[label=left:$12\;$]{};
            \vertex (21) at (2,0)[label=right:$\;n$]{};
            \vertex (22) at (2,1)[label=right:$\;13$]{};
            \vertex (23) at (2,2)[label=right:$\;12$]{};
            \path
                (11) edge (21)
                (13) edge[bend left=20] (23)
                (12) edge (21)
                (13) edge[dashed, bend right=20] (23)
            ;
            \draw (1,1.5) ellipse (1.4 and 1.1) [dashed];
        \end{tikzpicture}\\ \vspace{2em}
        \begin{tikzpicture}
            \node [below] (a) at (0,0){$1$};
            \node [below] (b) at (1,0){$n$};
            \node [below] (c) at (2,0){$2$};
            \node [below] (d) at (3,0){$3$};
            \node [above] (T1) at (3/2,3/2){$T_1$};
            \draw (0,0) -- (1/2,1/2);
            \draw (1,0) -- (1/2,1/2);
            \draw (1/2,1/2) -- (1,1);
            \draw (2,0) -- (5/2,1/2);
            \draw (3,0) -- (3/2,3/2);
            \draw (1/2,1/2) -- (3/2,3/2);
        \end{tikzpicture}
        \quad
        \begin{tikzpicture}
            \node [below] (a) at (0,0){$1$};
            \node [below] (b) at (1,0){$3$};
            \node [below] (c) at (2,0){$n$};
            \node [below] (d) at (3,0){$2$};
            \node [above] (T1) at (3/2,3/2){$T_2$};
            \draw (0,0) -- (1/2,1/2);
            \draw (1,0) -- (1/2,1/2);
            \draw (1/2,1/2) -- (1,1);
            \draw (2,0) -- (1,1);
            \draw (3,0) -- (3/2,3/2);
            \draw (1/2,1/2) -- (3/2,3/2);
        \end{tikzpicture}
        \quad
        \begin{tikzpicture}
            \vertex (11) at (0,0)[label=left:$n\;$]{};
            \vertex (12) at (0,1)[label=left:$23\;$]{};
            \vertex (13) at (0,2)[label=left:$12\;$]{};
            \vertex (21) at (2,0)[label=right:$\;n$]{};
            \vertex (22) at (2,1)[label=right:$\;13$]{};
            \vertex (23) at (2,2)[label=right:$\;12$]{};
            \path
                (11) edge (21)
                (13) edge[bend left=20] (23)
                (13) edge (21)
                (13) edge[dashed, bend right=20] (23)
            ;
            \draw (1,1.5) ellipse (1.4 and 1.1) [dashed];
        \end{tikzpicture}
        \caption{Case 3: 1 and 3 are closest in $T'_2$.  There are three subcases depending on $T'_1$.  The pair of vertices detached by $e_{12}$ are reconnected by $e_{2n}$.  Thus the precise grouping of the two components of $G_{T'_1,T'_2}^{H''}$ doesn't matter.}
        \end{subfigure}
\end{figure}
 
\begin{thm}[Laman's Theorem]
    Given a graph $H$, the following are equivalent:
    \begin{enumerate}
        \item\label{item:laman} $H$ is Laman,
        \item\label{item:henneberg} $H$ is Henneberg,
        \item\label{item:clade} there exist rooted binary trees $T_1$ and $T_2$ such that $G_{T_1,T_2}^H$ is a tree, and
        \item\label{item:minimallyRigid} $H$ is generically minimally rigid in the plane.
    \end{enumerate}
\end{thm}
\begin{proof}
    Proposition \ref{prop:tropicalMinRigidCondition} tells us that \eqref{item:clade} and \eqref{item:minimallyRigid} are equivalent.
    The implications \eqref{item:laman} $\implies$ \eqref{item:henneberg}, \eqref{item:henneberg} $\implies$ \eqref{item:clade}, and \eqref{item:clade} $\implies$ \eqref{item:laman}
    are Lemmas \ref{lem:LtoH}, \ref{lem:HtoT}, and \ref{lem:TtoL}, respectively.
\end{proof}

\section*{Acknowledgments}
\noindent
This material is based upon work supported by the National Science Foundation under Grant No. DMS-1439786 while the authors were in residence at the Institute for Computational and Experimental Research in Mathematics in Providence, RI, during the Fall 2018 semester.
The first author was also supported by an NSF Mathematical Sciences Postdoctoral Research Fellowship (DMS-1802902)
and is grateful to Seth Sullivant for many helpful conversations.
We also thank two anonymous referees for their careful reading and valuable suggestions.

\bibliographystyle{plain}
\footnotesize
\bibliography{tropicalLaman}

\end{document}